\documentclass[12pt,a4paper]{article}
\usepackage[margin=3.2cm]{geometry}
\usepackage{amsmath,amssymb,amsthm}
\usepackage{booktabs}
\usepackage[hidelinks]{hyperref}

\newtheorem{theorem}{Theorem}

\title{Counterexamples to a conjecture of Kamenetsky\\ on OEIS A173419}
\author{Rosario Patan\`e\thanks{SAMOVAR, T\'el\'ecom SudParis, Institut Polytechnique de Paris, Palaiseau, France. Email: rosario.patane@telecom-sudparis.eu}}
\date{}

\begin{document}
\maketitle

\begin{abstract}
Define $a(n)$ as the number corresponding to the shortest computation using only addition, subtraction, or multiplication, as A173419 states on the related OEIS page. In a comment to the sequence, D. Kamenetsky conjectured that $a(p)\ge a(p-1)$ for every prime. We show that the conjecture is false through four counterexamples below $5000$ that satisfy the condition $a(p)<a(p-1)$. They are exactly $3359$, $3623$, $4909$, and $4943$.
\end{abstract}

\section{Setting}

We fix $x_0=1$, and $x_m=n$, where every $x_k$ is the result of $x_i+x_j$, $x_i \cdot x_j$ or $x_i-x_j$ for some $0\le i,j < k$. We write $a(n)$ for the least such $m$, which corresponds to the OEIS A173419~\cite{oeis}. Greathouse introduced the sequence in OEIS in 2010. The sequence is the cost of the straight-line program of an integer. At the same time, bounds on $a(n!)$ are related to the intractability property of Hilbert's Nullstellensatz described in Shub and Smale~\cite{shubsmale}. The property satisfies $a(n)\le A005245(n)$, which Guy studied in his book~\cite[F26]{guy}.

The conjecture of D. Kamenetsky, in an OEIS comment of 26 December 2019, states \begin{equation}\label{eq:conj}
a(p)\ \ge\ a(p-1)\qquad\text{for every prime} \ p,
\end{equation}
which was verified for all primes $p<1800$, which corresponds to the range of the $b$-file then available.

\section{Results}

\begin{theorem}\label{thm:main}
The conjecture we presented in Eqn.~\eqref{eq:conj} is false. The prime numbers such that $p<5000$ with $a(p)<a(p-1)$ are $3359$, $3623$, $4909$, and $4943$ as reported in Table~\textup{\ref{tab:ce}}.
\end{theorem}

\begin{table}[ht]
\centering
\begin{tabular}{rrrll}
\toprule
$p$ & $a(p)$ & $a(p-1)$ & $p$ as & shortest computation for $p$\\
\midrule
$3359$ & $7$ & $8$ & $15^3-16$      & $1,2,4,16,15,225,3375,3359$\\
$3623$ & $8$ & $9$ & $29\cdot125-2$ & $1,2,4,5,25,29,125,3625,3623$\\
$4909$ & $7$ & $8$ & $17^3-4$       & $1,2,4,16,17,289,4913,4909$\\
$4943$ & $8$ & $9$ & $17\cdot291-4$ & $1,2,4,16,17,289,291,4947,4943$\\
\bottomrule
\end{tabular}
\caption{The four counterexamples under $5000$.}\label{tab:ce}
\end{table}

\begin{proof}
In each computation of Table~\ref{tab:ce}, every entry of the chain is obtained from the two previous entries by addition, multiplication, or subtraction. The first entry is $1$, and the last one is the prime~$p$. By counting the operations, $a(3359)\le 7$, $a(3623)\le 8$, $a(4909) \le 7$, and $a(4943) \le 8$.
  The enumeration with depth $8$ determines $a(n)$ for each $n<5000$ reachable with at most $8$ operations. This results in $a(3358)= 8$, $a(4908) =8$, and confirms the four upper bounds above as equalities. This establishes $a(n)\ge 9$ for all the others; hence, every pair $(p, p-1)$ is decided, except the unreachable ones. Among all pairs, nine $(p,p-1)$ with $p<5000$ remain undecided at depth $8$. They are $p=3323, 3803, 3929, 3947, 4139, 4583, 4783, 4871, 4987$. Neither member is reachable in $8$ operations; hence, $a(p), a(p-1) \ge 9$ for these.

Each of the eighteen numbers admits a computation of length $9$, which means $a(p)=a(p-1)=9$; none among the nine pairs violates Eqn.~\eqref{eq:conj}.
The same reasoning gives $a(3622)=a(4942)=9$: the lower bound from the enumeration, and the upper bound by appending one unitary subtraction to the computations of $3623$ and $4943$.
\end{proof}
\section{Computations}
The verification script (i) reproduces $a(1), \dots, a(108)$ and compares them with the values that are published at OEIS A173419, (ii) checks the four computations of Table \ref{tab:ce}, and
 (iii) reruns the search. The code is available at Github~\url{https://github.com/patan3saro/A173419-counterexample}. We report the results in the OEIS note.

 \textbf{AI Disclosure} The scripts were produced with the help of Claude. The repository's reproducibility guarantees verification and openness. Every result can be verified by hand, except the enumeration, which requires the code. The author found the results through the scripts and verified the computations.

\end{document}